\documentclass[11pt,a4paper,reqno]{amsart}
 \usepackage{amsfonts,amsmath,amscd,amssymb,amsbsy,amsthm,amstext,amsopn,amsxtra}
 \usepackage{color,fullpage,mathrsfs,verbatim,subfigure}
 \usepackage{stmaryrd}
 \usepackage{extarrows}
 \usepackage[all]{xy}
 \usepackage{bm}
 \usepackage{todonotes}
 \usepackage{hyperref}
 \usepackage{tikz}
 \usepackage{tikz-cd}
 \usetikzlibrary{shapes.geometric, calc,matrix,arrows,decorations.pathmorphing,arrows.meta}

\hypersetup{colorlinks,linkcolor=[rgb]{0,0,1},citecolor=[rgb]{1,0,0}}
 
\usepackage{todonotes}
 \numberwithin{equation}{section}
\newtheorem{theorem}{Theorem}[section]
\newtheorem{lemma}[theorem]{Lemma}
\newtheorem{proposition}[theorem]{Proposition}

\theoremstyle{definition}

\newtheorem{example}[theorem]{Example}

\theoremstyle{remark}
\newtheorem{remark}[theorem]{Remark}

\newtheorem{notation}[theorem]{Notation}

\newcommand{\kk}{\ensuremath{\Bbbk}} 
\renewcommand{\AA}{\ensuremath{\mathbb{A}}}

\newcommand{\NN}{\ensuremath{\mathbb{N}}} 
\newcommand{\PP}{\ensuremath{\mathbb{P}}}
\newcommand{\QQ}{\ensuremath{\mathbb{Q}}} 
\newcommand{\RR}{\ensuremath{\mathbb{R}}} 
\newcommand{\VV}{\ensuremath{\mathbb{V}}}
\newcommand{\ZZ}{\ensuremath{\mathbb{Z}}} 
\newcommand{\cM}{\mathcal{M}}

\def \O{\mathcal{O}}

\newcommand{\one}{\ensuremath{(\mathrm{i})}}
\newcommand{\two}{\ensuremath{(\mathrm{ii})}}

\renewcommand{\div}{\operatorname{div}}

\DeclareMathOperator{\End}{End}

\DeclareMathOperator{\git}{\!/\!\!/\!}

\DeclareMathOperator{\GL}{GL} 
 
\DeclareMathOperator{\Hom}{Hom}

\DeclareMathOperator{\Pic}{Pic}
 
\DeclareMathOperator{\rank}{rk}
 
\DeclareMathOperator{\Spec}{Spec}
 
\DeclareMathOperator{\supp}{supp}

\newcommand{\vv}{\bm{\mathrm v}}

\newcommand{\hd}{\operatorname{h}}
\newcommand{\Proj}{\operatorname{Proj}}
\newcommand{\Rep}{\operatorname{Rep}} 
\newcommand{\tl}{\operatorname{t}}
\newcommand{\Young}{\operatorname{Young}}

\def \TQ{Q^\prime} 
\def \p{\ell} 
\def \W{\mathcal{W}} 
\def \r{\underline{r}} 
\def \du{\underline{d}} 

\begin{document}

\title{Reconstructing toric quiver flag varieties from a tilting bundle}

\author{Alastair Craw} 
\address{Department of Mathematical Sciences, 
University of Bath, 
Claverton Down, 
Bath BA2 7AY, 
United Kingdom.}
\email{a.craw@bath.ac.uk / J.J.Green@bath.ac.uk}
\urladdr{http://people.bath.ac.uk/ac886/ and http://people.bath.ac.uk/jjg24/}
 
\author{James Green} 


\begin{abstract}
We prove that every toric quiver flag variety $Y$ is isomorphic to a fine moduli space of cyclic modules over the algebra $\End(T)$ for some tilting bundle $T$ on $Y$. This generalises the well known fact that $\mathbb{P}^n$ can be recovered from the endomorphism algebra of  $\bigoplus_{0\leq i\leq n} \mathcal{O}_{\mathbb{P}^n}(i)$. 
\end{abstract}

\maketitle

\section{Introduction}
 Nakajima~\cite[Section~3]{Nak96} introduced certain framed moduli spaces associated to a quiver, and the first author showed that these `quiver flag varieties' admit a tilting bundle \cite{Cra11}, generalising the construction of Beilinson~\cite{Bei78} and Kapranov~\cite{Kap88}. Here we extend this link further in the toric case by showing that every toric quiver flag variety can be reconstructed as a fine moduli space of cyclic modules over the endomorphism algebra of the tilting bundle.

 Before stating the main result we recall the construction and basic geometric properties of quiver flag varieties, also known as `framed quiver moduli'; references for this material include Nakajima~\cite[Section~3]{Nak96}, Reineke~\cite{Rei08} and Craw~\cite{Cra11}. Let $\kk$ be an algebraically closed field of characteristic zero and let $Q$ be a finite, connected, acyclic quiver with a unique source. Write $Q_0=\{0,1,\dots,\p\}$ for the vertex set, where $0$ is the source, and $Q_1$ for the arrow set, where for each $a\in Q_1$ we write $\hd(a)$ and $\tl(a)$ for the head and tail of $a$ respectively. Fix a dimension vector $\r=(r_i)\in \NN^{\p+1}$ satisfying $r_0=1$. The group $G:= \prod_{i=0}^{\p}\GL(r_i)$ acts by conjugation on the space $\Rep(Q,\r)=\bigoplus_{a\in Q_1} \Hom(\Bbbk^{r_{\tl(a)}}, \Bbbk^{r_{\hd(a)}})$ of representations of $Q$ of dimension vector $\r$, and we define the \emph{quiver flag variety} associated to the pair $(Q,\r)$ to be the GIT quotient
\[
Y:= \Rep(Q,\r)\git_\chi G
\]
for the special choice of linearisation $\chi:= (-\sum_{i=1}^{\p} r_i, 1,\dots, 1)\in G^{\vee}$. This GIT quotient is non-empty if and only if the inequality
\begin{equation}
\label{eqn:risi}
r_i\leq s_i:=\sum_{\{a\in Q_1 \mid \hd(a)=i\}} r_{\tl(a)}
\end{equation}
 holds for each $i>0$, in which case $Y$ is a smooth Mori Dream Space of dimension $\sum_{i=1}^\p r_i(s_i-r_i)$. In fact, $Y$ can be obtained from a tower of Grassmann-bundles
\begin{equation}
 \label{eqn:tower}
 Y:=Y_\p\longrightarrow Y_{\p-1}\longrightarrow \cdots \longrightarrow Y_1\longrightarrow Y_0=\Spec \kk,
\end{equation}
  where at each stage, $Y_i$ is isomorphic to the Grassmannian of rank $r_i$ quotients of a fixed locally-free sheaf of rank $s_i$ on $Y_{i-1}$; see \cite[Theorem~3.3]{Cra11}. Hereafter we assume that the inequality \eqref{eqn:risi} is strict for each $i>0$ to avoid degeneracy in the tower. 

Quiver flag varieties naturally carry a collection of vector bundles $\W_1,\dots, \W_\p$ that determine many of their algebraic invariants. Indeed, for  $i>0$, the Grassmann-bundle $Y_i$ over $Y_{i-1}$ carries a tautological quotient bundle $\mathcal{V}_i$ of rank $r_i$, and we write $\W_i:=\pi_i^*(\mathcal{V}_i)$ for the globally-generated bundle of rank $r_i$ on $Y$ obtained as the pullback under the morphism $\pi_i\colon Y\to Y_i$ in the tower. It follows from the construction that the invertible sheaves $\det(\W_1),\dots, \det(\W_\p)$ provide an integral basis for the Picard group of $Y$. More generally, the results of Beilinson~\cite{Bei78} and Kapranov~\cite{Kap88} extend to all quiver flag varieties as follows. Let $\Young(k,l)$ denote the set of Young diagrams with no more than $k$ columns and $l$ rows. Recall that for any vector bundle $\W$ of rank $r$ and for $\lambda \in \Young(k,l)$, we obtain a vector bundle $\mathbb{S}^{\lambda}\W$ whose fibre over each point is the irreducible $\GL(r)$-module of highest weight $\lambda$.
    
 \begin{theorem}[\cite{Cra11}]
 \label{thm:tilting}
 The vector bundle on $Y$ given by
  \begin{equation}
  \label{eqn:tilting}
  E:= \bigoplus_{1\leq i\leq \p, \; \lambda_i\in \Young(s_i-r_i,r_i)} \mathbb{S}^{\lambda_1}\W_1\otimes \dots \otimes \mathbb{S}^{\lambda_\p}\W_\p
  \end{equation}
 is a tilting bundle. In particular, the bounded derived category of coherent sheaves on $Y$ is equivalent to the bounded derived category of finite-dimensional modules over $A:=\End_{\mathcal{O}_Y}(E)$. 
  \end{theorem}

 This result answered affirmatively the question of Nakajima~\cite[Problem~3.10]{Nak96}.

\smallskip

 We now describe our main result. Work of Bergman-Proudfoot~\cite[Theorem~2.4]{BP08} compares any smooth projective variety admitting a tilting bundle to a fine moduli space of modules over the endomorphism algebra. To define the relevant moduli space for the tilting  bundle $E$ from \eqref{eqn:tilting}, list the indecomposable summands as $E_0, E_1,\dots, E_n$ with $E_0\cong \mathcal{O}_Y$, and consider the dimension vector $\vv:=(v_j)\in \NN^{n+1}$ satisfying $v_j:= \rank(E_j)$ for all $0\leq j\leq n$. For a special choice of `0-generated' stability condition $\theta$ (see Section~\ref{sec:reduction}), we consider the fine moduli space $\mathcal{M}(A,\vv,\theta)$ of isomorphism classes of $\theta$-stable $A$-modules of dimension vector $\vv$ that was constructed by King~\cite{Kin94} using GIT. Since each bundle $E_j$ is globally-generated, an observation of Craw--Ito--Karmazyn~\cite[Theorem~1.1]{CIK17} induces a universal morphism  
 \begin{equation}
 \label{eqn:universal}
 f_E\colon Y\longrightarrow \mathcal{M}(A,\vv,\theta),
 \end{equation}
 and in our case this is a closed immersion. In fact, \cite[Theorem~2.4]{BP08} implies that $f_E$ identifies $Y$ with a connected component of $\mathcal{M}(A,\vv,\theta)$, because $Y$ is smooth, $E$ is a tilting bundle, and our stability condition $\theta$ is `great'.

 Our main result concerns the special case when $r_i=1$ for all $1\leq i\leq \p$, in which case $G$ is an algebraic torus and therefore $Y$ is a toric variety; we call $Y$ a \emph{toric quiver flag variety}. The toric fan $\Sigma$ can be described directly in this case (see \cite[p1517]{CS08}), and $Y$ is a tower of projective space bundles via \eqref{eqn:tower}. We can say the following:
 
\begin{theorem}
\label{thm:main}
Let $Y$ be a toric quiver flag variety. The morphism $f_E\colon Y\to \mathcal{M}(A,\vv,\theta)$ from \eqref{eqn:universal} is an isomorphism.
\end{theorem}

 As a result, toric quiver flag varieties provide a new class of examples where the programme of Bergman-Proudfoot~\cite{BP08} can be carried out in full, enabling one to reconstruct the variety from the tilting bundle. The special case where $Y$ is isomorphic to projective space recovers the well-known result that $\mathbb{P}^n$ can be reconstructed from the tilting bundle $\bigoplus_{0\leq i\leq n} \mathcal{O}_{\mathbb{P}^n}(i)$ of Beilinson~\cite{Bei78}. Theorem~\ref{thm:main} therefore provides further evidence that toric quiver flag varieties provide good multigraded analogues of projective space.

\subsection*{Acknowledgements} We thank the anonymous referee for a number of helpful comments. The first author was supported in part by EPSRC grant EP/J019410/1, and the second author is supported by a doctoral studentship from EPSRC. 
 
\section{The reduction step}
\label{sec:reduction}
Our assumption gives $1 = r_i = \rank(\mathcal{W}_i)$ for $1\leq i\leq \p$, so the tilting bundle from \eqref{eqn:tilting} is simply the direct sum of line bundles
 \begin{equation}
 \label{eqn:torictilt}
 E=\bigoplus_{1\leq i \leq \p, \ 0\leq m_i< s_i } \W_1^{\otimes m_1}\otimes \cdots \otimes \W_\p^{\otimes m_\p}
 \end{equation}
 on $Y$. Set $n+1:=\prod_{1\leq i\leq \ell} s_i$, and list the indecomposable summands from \eqref{eqn:torictilt} as $E_0, \dots, E_{n}$ with $E_0\cong \mathcal{O}_Y$. Consider the endomorphism algebra $A:= \End_{\mathcal{O}_Y}(E)$ and the dimension vector $\vv:=(1, \dots, 1)\in \ZZ^{n+1}$.
 
The moduli space that features in Theorem~\ref{thm:main} is an example of those constructed originally by King~\cite{Kin94}. To introduce our choice of stability condition, first set $\theta^\prime=(-n,1,1,\dots,1)\in \Hom(\ZZ^{n+1},\QQ)$. An $A$-module $M=\bigoplus_{0\leq j\leq n} M_j$ of dimension vector $\vv$ is $\theta^\prime$-stable iff $M$ is generated as an $A$-module by any nonzero element of $M_0$; any such stability condition is called \emph{0-generated}. Since $\vv$ is primitive and since every $\theta^\prime$-semistable $A$-module of dimension vector $\vv$ is $\theta^\prime$-stable (see, for example, \cite[Proof of Proposition~3.8]{CS08}), King~\cite[Proposition~5.3]{Kin94} constructs the fine moduli space $\mathcal{M}(A,\vv,\theta^\prime)$ of isomorphism classes of $\theta^\prime$-stable $A$-modules of dimension vector $\vv$ as a GIT quotient. In particular, $\mathcal{M}(A,\vv,\theta^\prime)$ comes with an ample bundle $\mathcal{O}(1)$. Let $k\geq 1$ denote the smallest positive integer such that $\mathcal{O}(k)$ is very ample. Then $\theta:=k\theta'$ is also a 0-generated stability condition, and we consider the fine moduli space $\cM(A,\vv, \theta)$ of $\theta$-stable $A$-modules of dimension vector $\vv$; this moduli space is the `multigraded linear series' of $E$ in the sense of \cite[Definition~2.5]{CIK17}.
 
 Since each indecomposable summand of $E$ from \eqref{eqn:torictilt} is globally-generated, we deduce from \cite[Theorem~2.6]{CIK17} that the universal property of $\cM(A,\vv, \theta)$ gives a morphism 
 \[
 f_E\colon Y\longrightarrow \mathcal{M}(A,\vv,\theta)
 \]
 and, moreover, $f_E$ is a closed immersion because the line bundle $\bigotimes_{0\leq j\leq n} E_j$ is very ample. This puts us in the situation studied by Craw--Smith~\cite{CS08}, where it is possible to give an explicit GIT quotient description for both the moduli space $\mathcal{M}(A,\vv,\theta)$ and the image of the universal morphism $f_E$. Theorem~\ref{thm:main} will follow once we prove that these two GIT quotients coincide.
 
 \medskip
 
To describe $\mathcal{M}(A,\vv,\theta)$ as a GIT quotient, we first present the algebra $A=\End_{\mathcal{O}_Y}(E)$ using the bound quiver of sections $(\TQ,R)$ as follows. The quiver $\TQ$ has vertex set $\TQ_0=\{0,1,\dots,n\}$ and an arrow from vertex $i$ to $j$ for each \emph{irreducible}, torus-invariant section of $E_{j}\otimes E_{i}^{-1}$, i.e., the corresponding homomorphism from $E_i$ to $E_j$ does not factor through some $E_k$ with $k \neq i, j$. To each arrow $a\in \TQ_1$ we associate the corresponding torus-invariant `labeling divisor' $\div(a)\in \mathbb{N}^{\Sigma(1)}$, where $\Sigma(1)$  denotes the set of rays of the fan of $Y$. The two-sided ideal
 \[
 R := \left(p-q\in \kk \TQ \mid p,q \text{ share the same head, tail and labeling divisor}\right)
 \]
 in $\kk \TQ$ satisfies $A\cong \kk \TQ/R$ (see \cite[Proposition~3.3]{CS08}). Denote the coordinate ring of $\mathbb{A}^{\TQ_1}_\kk$ by $\kk[y_a]$, where $a$ ranges over $\TQ_1$. The ideal $R$ in the noncommutative ring $\kk \TQ$ determines an ideal in $\kk[y_a]$ given by
  \begin{equation}
 \label{eqn:IR}
 I_R:=\left(\prod_{a\in \supp(p)} y_a - \prod_{a\in \supp(q)} y_a \in \kk[y_a]\mathrel{\Big|} \begin{array}{c} p, q \text{ share the same head,} \\ \text{tail and labeling divisor} \end{array}\right),
 \end{equation}
 where the support of a path $\supp(p)$ is simply the set of arrows that make up the path. This ideal is homogeneous with respect to the action of $T:= \prod_{0\leq j\leq n} \GL(1)$ by conjugation.  It now follows directly from the definition of King~\cite{Kin94} that 
 \begin{equation}
 \label{eqn:GITmoduli}
 \mathcal{M}(A,\vv,\theta) := \VV(I_R)\git_\theta T := \Proj \bigoplus_{k\geq 0}\big(\kk[y_a]/I_R)_{k\theta},
 \end{equation}
 where $\big(\kk[y_a]/I_R)_{k\theta}$ denotes the $k\theta$-graded piece. In fact, \cite[Proposition~3.8]{CS08} implies that $\mathcal{M}(A,\vv,\theta)$ is the geometric quotient of $\VV(I_R)\setminus \VV(B_{\TQ})$ by the action of $T$, where 
\begin{equation}
\label{eqn:BQ}
B_{\TQ}:=\bigcap_{j=1}^{n}\big( y_a \in \kk[y_a] \mathrel{\big|} \hd(a)=j\big)
\end{equation}
 is the irrelevant ideal in $\kk[y_a]$ that cuts out the $\theta$-unstable locus in $\AA_\kk^{\TQ_1}$.
 
 Our task is to compare \eqref{eqn:GITmoduli} with the GIT quotient description of the image of $f_E$. For this, define a map $\pi\colon \ZZ^{\TQ_1}\rightarrow \ZZ^{\TQ_0}\oplus \ZZ^{\Sigma(1)}$ by setting $\pi(\chi_a)=(\chi_{\hd(a)}-\chi_{\tl(a)},\div(a))$, where $\chi_a$ for $a\in \TQ_1$ and $\chi_i$ for $i\in \TQ_0$ denote the characteristic functions. The $T$-homogeneous ideal
 \begin{equation}
 \label{eqn:IQ}
 I_{\TQ}:=\big(y^u-y^v\in \kk[y_a] \mid u-v \in \ker(\pi)\big)
 \end{equation}
 contains $I_R$ from \eqref{eqn:IR}, and \cite[Proposition~4.3]{CS08} establishes that the image of the universal morphism $f_E$ is isomorphic to the geometric quotient of $\VV(I_{\TQ})\setminus \VV(B_{\TQ})$ by the action of $T$. 
 
 \begin{proposition}
 \label{prop:reduction}
 Suppose that the $T$-orbit of every closed point of $\VV(I_{R})\setminus \VV(B_{\TQ})$ contains a closed point of $\VV(I_{\TQ})\setminus \VV(B_{\TQ})$. Then Theorem~\ref{thm:main} holds.
 \end{proposition}
 \begin{proof}
 The inclusion $\VV(I_{\TQ})\subseteq \VV(I_R)$ always holds, and the assumption ensures that $\VV(I_{R})\git_\theta T\subseteq \VV(I_{\TQ})\git_\theta T$, so the closed immersion $f_E$ is surjective.
 \end{proof}
 
 In Section~\ref{sec:proof} we prove that the assumption of Proposition~\ref{prop:reduction} holds for every toric quiver flag variety $Y$. To illustrate the strategy, we recall the following well-known construction of $\mathbb{P}^n$ using Beilinson's tilting bundle.

\begin{example}
\label{exa:beilinson}
 For the acyclic quiver $Q$ with vertex set $Q_0=\{0,1\}$ and $n+1$ arrows from $0$ to $1$,  the toric quiver flag variety $Y$ is isomorphic to $\mathbb{P}^n$ and the quiver of sections $\TQ$ for the tilting bundle $\bigoplus_{0\leq i\leq n} \mathcal{O}_{\mathbb{P}^n}(i)$ is shown in Figure~\ref{fig:BeilinsonQuiver}; note that $Q$ is a subquiver of $\TQ$.
	\begin{figure}[!ht]
		\begin{center}
			\begin{tikzpicture}
			\tikzset{vertex/.style = {shape=circle,draw,minimum size=3.3em}}
			\tikzset{edge/.style = {->,> = latex'}}
			\node[vertex,thick] (0) at  (0,0) {$0$};
			\node[vertex,thick] (1) at  (2.5,0) {$1$};
			\node[vertex,thick] (p-1) at  (7.5,0) {$n-1$};
			\node[vertex,thick] (p) at  (10,0) {$n$};
			\draw (1.2,0.5) node {$n+1$};
			\draw (8.75,0.5) node {$n+1$};
			\draw (5,0) node {\Large{$\cdots$}};
			\draw[edge] (0) -- (1);
			\draw[edge] (0.59,-0.2) -- (1.9,-0.2);
			\draw[edge] (0.59,0.2) -- (1.9,0.2);
			\draw[] (3.12,0) -- (4,0);
			\draw[] (3.11,-0.2) -- (4,-0.2);
			\draw[] (3.11,0.2) -- (4,0.2);
			\draw[edge] (6,0) -- (6.86,0);
			\draw[edge] (6,-0.2) -- (6.89,-0.2);
			\draw[edge] (6,0.2) -- (6.89,0.2);
			\draw[edge] (p-1) to (p);
			\draw[edge] (8.09,-0.2) -- (9.4,-0.2);
			\draw[edge] (8.09,0.2) -- (9.4,0.2);
			\end{tikzpicture}
		\end{center}
		\caption{The tilting quiver of $\PP^n$.}
         \label{fig:BeilinsonQuiver}
	\end{figure}	
For each $1\leq m\leq n$ and each ray $\rho\in \Sigma(1)$ in the fan of $\mathbb{P}^n$ defining a torus-invariant divisor $D_\rho$, let $a_\rho^m$ denote the arrow with head at $m$ and labeling divisor $\div(a_\rho^m) = D_\rho$. Writing $y_\rho^m \in \kk[y_a]$ for the variable associated to the arrow $a_\rho^m$, we have 
  \begin{equation}
  \label{eqn:IRBeilinson}
  I_R=\big(y^{m+1}_\sigma y^m_\rho-y^{m+1}_\rho y^m_\sigma \in \kk[y_a] \mathrel{\big|} 1\leq m\leq n-1;\; \rho,\sigma\in \Sigma(1)\big).
  \end{equation}
 We claim that a point $(w_\rho^m)\in \VV(I_R)\setminus \VV(B_{\TQ})\subset \mathbb{A}_\kk^{n(n+1)}$ lies in the same $T$-orbit as the point $(v_\rho^m)$ with components $v_\rho^m:= w_\rho^1$ for all $1\leq m\leq n$ and $\rho\in \Sigma(1)$. Clearly $(v_\rho^m)\in \VV(I_{\TQ})\setminus \VV(B_{\TQ})$, so the claim and Proposition~\ref{prop:reduction} show that Theorem~\ref{thm:main} holds for $\mathbb{P}^n$.
  
  To prove the claim, note that since $(w_\rho^m)\not\in \VV(B_{\TQ})$, the $T$-action allows us to assume that for all $1\leq m\leq n$ there exists $\rho(m)\in \Sigma(1)$ such that $w^m_{\rho(m)}= 1$. Then $v^1_{\rho(1)} = 1$, and \eqref{eqn:IRBeilinson} implies that $w^2_{\rho(1)} v^1_\rho=w^2_\rho$ for all $\rho\in \Sigma(1)$. The case $\rho=\rho(2)$ gives $w^2_{\rho(1)}= (v^1_{\rho(2)})^{-1}= (w^1_{\rho(2)})^{-1}$, so 
   \[
   w^2_\rho = v^1_\rho(w^1_{\rho(2)})^{-1}= w^1_\rho(w^1_{\rho(2)})^{-1} \quad \text{for all }\rho\in \Sigma(1).
   \]
 Let the one-dimensional subgroup $\kk^\times\subset T$ scale by $w^1_{\rho(2)}$ at vertex $2$ to obtain a point in the same $T$-orbit as $(w_\rho^m)$ whose components agree with those of $(v_\rho^m)$ for $m=1,2$. Repeating at each successive vertex shows that $(v_\rho^m)$ and $(w_\rho^m)$ lie in the same $T$-orbit as claimed.
\end{example}
 
\section{The tilting quiver}
\label{sec:TQ}
 Before establishing that the assumption of Proposition~\ref{prop:reduction} holds for every toric quiver flag variety, we describe the \emph{tilting quiver} $\TQ$ in detail (see Example~\ref{exa:7arrows}).

 For the vertex set $\TQ_0$, recall that the line bundles $\W_1, \dots, \W_\p$ provide an integral basis for $\Pic(Y)\cong \ZZ^{\ell}$. Since $\TQ_0$ is defined by the summands $\W_1^{\otimes m_1}\otimes \cdots \otimes \W_\p^{\otimes m_\p}$ of the tilting bundle $E$ from \eqref{eqn:torictilt}, it is convenient to realise $\TQ_0$ as the set of lattice points of a cuboid in $\ZZ^\ell\otimes_{\ZZ}\RR$ of dimension $\p$ with side lengths $s_1-1,\dots,s_{\p}-1$. We label the vertex for $\W_1^{\otimes m_1}\otimes \cdots \otimes \W_\p^{\otimes m_\p}$ by the corresponding lattice point $(m_1,\dots,m_\p)\in\ZZ^\p$, giving 
 \[
 \TQ_0=\{(m_1,\dots,m_\p)\in \ZZ^{\ell} \mid 0\leq m_i<s_i\}.
 \]
 We introduce a total order on $\TQ_0$: for $k=(k_1,\dots,k_\p),m=(m_1,\dots,m_\p)\in\TQ_0$, write $k<m$ if $k_i<m_i$ for the largest index $i$ satisfying $k_i\neq m_i$.
 
 For the arrow set $\TQ_1$, note first that $Q$ is the quiver of sections of $\{\mathcal{O}_Y,\W_1,\dots,\W_\p\}$, so the arrows in $Q$ correspond precisely to the torus-invariant prime divisors in $Y$ \cite[Remark~3.9]{CS08}. For $\rho\in \Sigma(1)$ we write $a_\rho\in Q_1$ for the arrow corresponding to the divisor of zeros $D_\rho$ of a torus-invariant section of $\W_{\hd(a_\rho)}\otimes \W_{\tl(a_\rho)}^{-1}$. Each $a_\rho$ may be regarded as an arrow in $\TQ$, so we may identify $Q$ with a complete subquiver of $\TQ$ that we call the \emph{base quiver} in $\TQ$. More generally, translating each $a_\rho$ around the cuboid described in the preceding paragraph (so that the head and tail lie in $\TQ_0$) produces arrows in $\TQ$ that we denote $a_\rho^m\in \TQ_1$ for $m=\hd(a_\rho^m)$ and $D_\rho = \div(a_\rho^m)$. In fact, we have the following:
  
\begin{lemma}
\label{lem:tiltingarrows}
 Every arrow $a\in \TQ_1$ is of the form $a=a^m_\rho$, where $m=\hd(a)$ and $D_\rho=\div(a)$. 
\end{lemma}
 
\begin{proof}
For $a\in \TQ_1$, write $\hd(a)=m=(m_1,\dots,m_\p)$ and $\tl(a)=m^\prime=(m^\prime_1,\dots,m^\prime_\p)$, so $\div(a)$ is the divisor of zeros of a section of $\bigotimes_{1\leq i\leq \ell}\mathcal{W}_i^{\otimes(m_i-m^\prime_i)}$. In terms of prime divisors, we have 
\[
\div(a)=\sum_{\rho\in \Sigma(1)} \lambda_{\rho} D_{\rho} \quad\text{ for }\lambda_{\rho}\in \NN.
\]
 Let $1\leq k\leq \p$ be the largest value such that  $\lambda_\rho\neq 0$ for some $\rho\in \Sigma(1)$ satisfying $k=\hd(a_\rho)\in Q_0$. Note that $0\leq m^\prime_k<m_k$, and moreover, $j:=\tl(a_\rho)<k$. Since $\div(a)$ is irreducible, translating $a_\rho$ so that the tail is at vertex $m^\prime$ forces the head to lie outside the cuboid, giving $m^\prime_j=0$ or $m^\prime_k=s_k-1$; similarly, translating $a_\rho$ so that the head is at $m$ forces the tail to lie outside the cuboid, giving $m_j=s_j-1$ or $m_k=0$. Since $0\leq m^\prime_k<m_k$, both $m^\prime_j=0$ and $m_j=s_j-1$ must hold, so $m^\prime_j<m_j$. As a result, there must exist $\sigma\in \Sigma(1)$ satisfying $\lambda_{\sigma}\neq 0$ for $j=\hd(a_{\sigma})$. If we set $i:=\tl(a_{\sigma})$ and repeat the argument above, we deduce that $m^\prime_i<m_i$. Continuing in this way, we eventually find $\tau\in \Sigma(1)$ such that $\lambda_\tau\neq 0$ with $\hd(a_\tau) = 1$ and $\tl(a_\tau)=0$. But then $0=m^\prime_1<m_1=s_1-1$, so we can place a translation of $a_\tau$ with head at $m$ and tail in the cuboid (or tail at $m^\prime$ and head in the cuboid). This shows $\div(a)$ is reducible, a contradiction.
\end{proof}

 \begin{remark}
 \label{rem:ei}
 Since $Q$ is the quiver of sections of $\{\mathcal{O}_Y,\W_1,\dots,\W_\p\}$, the vertices of the base quiver are the vertices $e_0, e_1, \dots, e_\ell\in \TQ_0\subset \ZZ^{\ell}$, where $e_i$ denotes the $i^{\text{th}}$ standard basis vector for $i>0$, and where $e_0:=(0,\dots,0)$.  \end{remark}

The next example illustrates how the base quiver sits inside $\TQ$.

\begin{example}
\label{exa:7arrows}
 The quiver $Q$ shown in Figure~\ref{fig:basequiverQ} defines the toric quiver flag variety $Y = \PP_{Z}(\O(1,0)\oplus \O(0,1))$ where $Z=\PP_{\PP^2}(\O\oplus \O(1))$; the colours of the arrows indicate the distinct labeling divisors. We have $s_1=3$ and $s_2=s_3=2$, so the tilting quiver $\TQ$ has $12$ vertices shown in Figure~\ref{fig:tiltingquiverTQ} using the ordering described above. 
   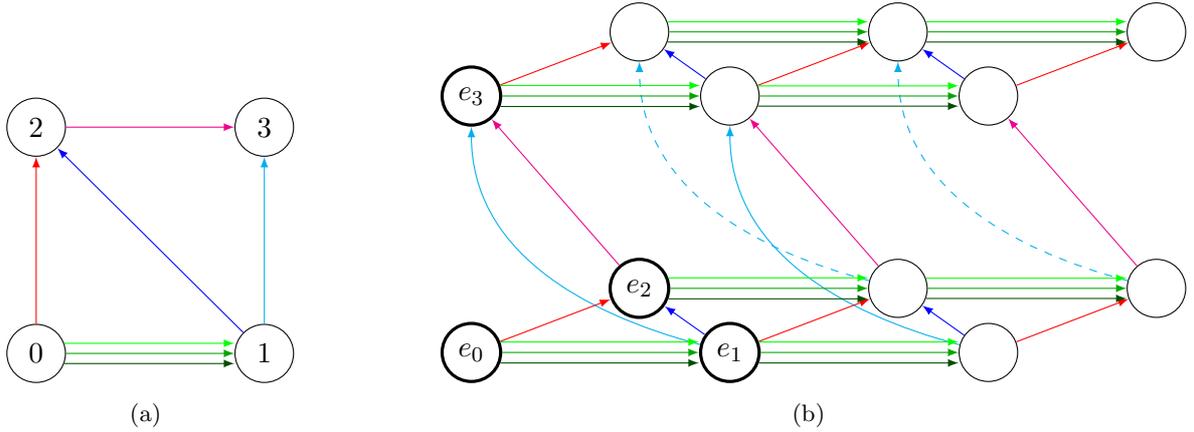
\begin{figure}[!ht]
   \centering
      \subfigure[]{
      \label{fig:basequiverQ}
       \begin{tikzpicture}[xscale=1.5,yscale=1.5]
			\tikzset{vertex/.style = {shape=circle,draw,minimum size=2em}}
			\tikzset{>=latex}
			\node[vertex] (A) at (0,0) {0};
			\node[vertex] (B) at (2,0) {1};
			\node[vertex] (C) at (0,2) {2};
			\node[vertex] (D) at (2,2) {3};
			\draw [->][green] (A.20) -- (A.20-|B.west);
			\draw [->][black!35!green] (A) -- (B);
			\draw [->][black!70!green] (A.-20) -- (A.-20-|B.west);
			\draw [->][red] (A) -- (C);
			\draw [->][blue] (B) -- (C);
			\draw [->][cyan] (B) -- (D);
			\draw [->][magenta] (C) -- (D);
			\end{tikzpicture} 
       }
      \qquad  \qquad
      \subfigure[]{
       \label{fig:tiltingquiverTQ}
            \begin{tikzpicture}[xscale=1.7,yscale=1.7]
			\tikzset{vertex/.style = {shape=circle,draw,minimum size=2em}}
			
			\tikzset{>=latex}
			
			\node[vertex, very thick] (A) at (0,0) {$e_0$};
			\node[vertex, very thick] (B) at (2,0) {$e_1$};
			\node[vertex, very thick] (C) at (1.3,0.5) {$e_2$};
			\node[vertex, very thick] (D) at (0,2) {$e_3$};
			\node[vertex] (E) at (4,0) {};
			\node[vertex] (F) at (3.3,0.5) {};
			\node[vertex] (G) at (2,2) {};
			\node[vertex] (H) at (1.3,2.5) {};
			\node[vertex] (I) at (5.3,0.5) {};
			\node[vertex] (J) at (4,2) {};
			\node[vertex] (K) at (3.3,2.5) {};
			\node[vertex] (L) at (5.3,2.5) {};
			
			\draw [->][cyan] (B) to [out=165, in=270] (D);
			\draw [->][cyan] (E) to [out=165, in=270] (G);
			\draw [dashed, ->][cyan] (F) to [out=165, in=270] (H);      
         	\draw [dashed, ->][cyan] (I) to [out=165, in=270] (K);
			
			\draw [->][magenta] (C) -- (D);
			\draw [->][magenta] (F) -- (G);
			\draw [->][magenta] (I) -- (J);
			
			\draw [->][red] (A) -- (C);
			\draw [->][red] (B) -- (F);
			\draw [->][red] (E) -- (I);
			\draw [->][red] (D) -- (H);
			\draw [->][red] (G) -- (K);
			\draw [->][red] (J) -- (L);
			
			\draw [->][blue] (B) -- (C);
			\draw [->][blue] (E) -- (F);
			\draw [->][blue] (G) -- (H);
			\draw [->][blue] (J) -- (K);
			
			\draw [->][green] (A.20) -- (A.20-|B.west);
			\draw [->][black!35!green] (A) -- (B);
			\draw [->][black!70!green] (A.-20) -- (A.-20-|B.west);
			
			\draw [->][green] (B.20) -- (B.20-|E.west);
			\draw [->][black!35!green] (B) -- (E);
			\draw [->][black!70!green] (B.-20) -- (B.-20-|E.west);
			
			\draw [->][green] (C.20) -- (C.20-|F.west);
			\draw [->][black!35!green] (C) -- (F);
			\draw [->][black!70!green] (C.-20) -- (C.-20-|F.west);   
            
			\draw [->][green] (F.20) -- (F.20-|I.west);
			\draw [->][black!35!green] (F) -- (I);
			\draw [->][black!70!green] (F.-20) -- (F.-20-|I.west);
			
			\draw [->][green] (D.20) -- (D.20-|G.west);
			\draw [->][black!35!green] (D) -- (G);
			\draw [->][black!70!green] (D.-20) -- (D.-20-|G.west);
			
			\draw [->][green] (G.20) -- (G.20-|J.west);
			\draw [->][black!35!green] (G) -- (J);
			\draw [->][black!70!green] (G.-20) -- (G.-20-|J.west);
			
			\draw [->][green] (H.20) -- (H.20-|K.west);
			\draw [->][black!35!green] (H) -- (K);
			\draw [->][black!70!green] (H.-20) -- (H.-20-|K.west);
			
			\draw [->][green] (K.20) -- (K.20-|L.west);
			\draw [->][black!35!green] (K) -- (L);
			\draw [->][black!70!green] (K.-20) -- (K.-20-|L.west);
			
			\end{tikzpicture}
            }
           \caption{Quivers for $Y$: (a) original quiver $Q$; (b) tilting quiver $\TQ$.}
           \label{fig:bothQuivers}
  \end{figure}
  Note that the base quiver is the complete subquiver of $\TQ$ whose vertices are shown in bold in Figure~\ref{fig:tiltingquiverTQ}. The colour of each arrow of $\TQ$ is determined by its unique translate arrow from the base quiver. 
\end{example}

\section{Proof of Theorem~\ref{thm:main}}
\label{sec:proof}
 In light of Lemma~\ref{lem:tiltingarrows}, each point of $\AA^{\TQ_1}_\kk$ is a tuple $(w_\rho^m)$ where $w_\rho^m\in \kk$ for $\rho\in \Sigma(1)$ and for all relevant $m\in \TQ_0$. Motivated by Example~\ref{exa:beilinson}, we associate to $(w_\rho^m)\in \AA^{\TQ_1}_\kk$ an auxiliary point $(v_\rho^m)\in \VV(I_{\TQ})\subseteq \mathbb{A}^{\TQ_1}_\kk$ whose components satisfy
\begin{equation}
\label{eqn:vrhom}
 v_\rho^m:=w_\rho \text{ for }\rho\in \Sigma(1)\text{ and all relevant }m\in \TQ_0,
\end{equation}
 where for $\rho\in \Sigma(1)$ we write $w_\rho\in \kk$ for the component of the point $(w_\rho^m)$ corresponding to the unique arrow $a_\rho$ in the base quiver satisfying $\div(a_\rho)=D_\rho$. 

\begin{lemma}
If $(w_\rho^m)\not\in \VV(B_{\TQ})$, then $(v_\rho^m)\not\in \VV(B_{\TQ})$.
\end{lemma}
\begin{proof}
Fix $m=(m_1,\dots,m_\p)\in \TQ_0$ and let $1\leq j\leq \p$ be minimal such that $m_j\neq0$. Then for all $\rho$ satisfying $\hd(a_\rho)=j\in Q_0$, the arrow $a_\rho^m$ obtained by translating $a_\rho$ until the head lies at $m$ is an arrow of $\TQ$. At least one of the values $\{w_\rho \mid \hd(a_\rho)=j\}$ is nonzero by assumption, and hence for this value of $\rho$ we have $v_\rho^m=w_\rho\neq 0$ as required.
\end{proof}

 We now establish notation for the proof of Theorem~\ref{thm:main}. For any vertex $k=(k_1,\dots,k_\p)\in \TQ_0$, let $(\TQ(k),R(k))$ denote the bound quiver of sections of the line bundles $\W_1^{\otimes m_1}\otimes \cdots \otimes \W_\p^{\otimes m_\p}$ on $Y$ with $(m_1,\dots,m_\p)\leq k$. Explicitly, $\TQ(k)$ is the complete subquiver of $\TQ$ with vertex set $\TQ(k)_0:=\{m\in\TQ_0 \mid m\leq k \}$, and the ideal of relations $R(k):=\kk\TQ(k)\cap R$ satisfies 
 \[
 \frac{\kk \TQ(k)}{R(k)}\cong \End\Bigg(\bigoplus_{(m_1,\dots,m_\p)\leq k} \W_1^{\otimes m_1}\otimes \cdots \otimes \W_\p^{\otimes m_\p}\Bigg).
 \]
 As in Section~\ref{sec:reduction}, the coordinate ring $\kk[y_\rho^m \mid \rho\in \Sigma(1), m\leq k]$
 of the affine space $\mathbb{A}^{\TQ(k)_1}_\kk$ contains ideals $I_{R(k)}, B_{\TQ(k)}$ and $I_{\TQ(k)}$ defined as in equations \eqref{eqn:IR}, \eqref{eqn:BQ} and \eqref{eqn:IQ} respectively, each of which is homogeneous with respect to the action of $T(k):=\prod_{0\leq i\leq k} \GL(1)$ by conjugation. The projection onto the coordinates indexed by arrows $a_\rho^m$ satisfying $m\leq k$, denoted
 \begin{equation}
 \label{eqn:piN}
 \pi_k\colon \mathbb{A}^{\TQ_1}_\kk\longrightarrow \mathbb{A}^{\TQ(k)_1}_\kk,
 \end{equation}
 is equivariant with respect to the actions of $T$ and $T(k)$. Notice that $\pi_k(\VV(I_R))\subseteq \VV(I_{R(k)})$, $\pi_k(\VV(B_{\TQ}))\subseteq \VV(B_{\TQ(k)})$ and $\pi_k(\VV(I_{\TQ}))\subseteq \VV(I_{\TQ(k)})$.

\begin{proof}[Proof of Theorem~\ref{thm:main}]
 Fix a point $w=(w_\rho^m)\in \VV(I_R)\setminus \VV(B_{\TQ})$ and the corresponding point $v=(v^m_\rho)\in \VV(I_{\TQ})\setminus \VV(B_{\TQ})$ whose components are defined in equation \eqref{eqn:vrhom}.  Since $w\not\in \VV(B_{\TQ})$, the action of $T$ enables us to assume that for all $m\in \TQ_0$ there exists $\rho(m)\in \Sigma(1)$ such that $w_{\rho(m)}^m=1$. In particular, $v_{\rho(m)}^m=1$ for all relevant $m\in \TQ_0$.  Now, for $0\leq k\leq (s_1-1,\dots,s_\p-1)$, the morphism $\pi_k$ from \eqref{eqn:piN} sends the points $w$ and $v$ to 
 \[
 \pi_k(w)\in \VV(I_{R(k)})\setminus \VV(B_{\TQ(k)}) \quad \text{and}\quad \pi_k(v)\in \VV(I_{\TQ(k)})\setminus \VV(B_{\TQ(k)})
 \]
 respectively. We claim that $\pi_k(v)$ lies in the $T(k)$-orbit of $\pi_k(w)$. Given the claim, the special case $k=(s_1-1,\dots,s_\p-1)$ shows that the point $v$ lies in the $T$-orbit of the point $w$, so Theorem~\ref{thm:main} follows immediately from Proposition~\ref{prop:reduction}. 
 
 We prove the claim by induction on the vertex $k = (k_1,\dots,k_\p)$ using the total order on $\TQ_0$ from Section~\ref{sec:TQ}. The case $k=e_0$ is immediate, and for $(1,0,\dots,0)\leq k\leq (s_1-1,0,\dots,0)$ the claim follows from Example~\ref{exa:beilinson}; hereafter we assume that $\ell\geq2$. Suppose the claim holds for all $m<k$, so we may assume that $w_\rho^m = w_\rho$ for all $m<k$. It is enough to show for all $\rho\in \Sigma(1)$, that $w_{\rho(k)}\neq 0$ and
 \begin{equation}
\label{eqn:wrhoN}
w_{\rho}^k=w_{\rho}(w_{\rho(k)})^{-1},
\end{equation}
 because then we may let the one-dimensional subgroup $\kk^\times\subset T(k)$ scale by $w_{\rho(k)}$ at vertex $k$. Before establishing the claim \eqref{eqn:wrhoN}, we introduce some notation that we use in the proof.
  
 \begin{notation}
 \begin{enumerate}
 \item Recall from Section~\ref{sec:TQ} that vertices of the tilting quiver $Q^\prime$ are elements $k = (k_1,\dots,k_\ell)$ in the lattice $\ZZ^\ell$, so $k_i\in \ZZ$ for $1\leq i\leq \ell$. Note also (see Remark~\ref{rem:ei}) that the standard basis vectors $e_1,\dots, e_\ell$ of $\ZZ^\ell$ denote certain vertices of $Q^\prime$. This notation is standard and we hope that no confusion arises in what follows.
 \item It is convenient to distinguish certain elements of $Q_0$ and $\ZZ^\p$.
 \begin{itemize}
 \item First we distinguish certain elements of the vertex set $Q_0=\{0,1,\dots,\p\}$ of the original quiver. For the ray $\rho(k)$ appearing in \eqref{eqn:wrhoN}, define $0\leq \alpha<\beta\leq \p$ by 
 \[
 \alpha:=\tl(a_{\rho(k)})\quad\text{and}\quad \beta:=\hd(a_{\rho(k)}), 
 \]
 where $a_{\rho(k)}$ is the arrow in the original quiver $Q$ satisfying $\div(a_{\rho(k)})=D_{\rho(k)}$. Also, let $1\leq \delta\leq \p$ be minimal such that the induction vertex $k=(k_1,\dots,k_\p)$ satisfies $k_\delta\neq 0$, and define $0\leq \gamma<\delta$ by setting
  \[
 \gamma:=\tl(a_{\rho(e_\delta)}).
 \]
 Minimality of $\delta$ implies that either $\gamma=0$ or $k_\gamma=0$ and, moreover, that $\delta\leq \beta$.
 \item Next we introduce certain elements of $\ZZ^\p$. For any ray $\rho\in\Sigma(1)$, define 
 \[
 \du(\rho):=e_{\hd(a_\rho)}-e_{\tl(a_\rho)}\in  \ZZ^\p,
 \]
 where $a_\rho$ is the arrow in the original quiver satisfying $\div(a_\rho)=D_\rho$ (recall that $e_0:=0$). In particular, by the previous bullet point we have
 \[
 \du(\rho(k))=e_\beta-e_\alpha \quad \text{and}\quad \du(\rho(e_\delta))=e_\delta-e_\gamma.
 \]
 \end{itemize}
 \end{enumerate}
 \end{notation}
 
 \medskip
 
 We now return to the proof of the claim \eqref{eqn:wrhoN}, treating the cases $\delta<\beta$ and $\delta=\beta$ separately.
 
 \medskip
 
\noindent \textsc{Case 1:} Suppose first that $\delta<\beta$. In this case we proceed in three steps:

 \smallskip
\indent \textsc{Step 1:} Show that equation \eqref{eqn:wrhoN} holds for $\rho=\rho(e_\delta)$ when $\gamma=\alpha=0$ or $\gamma\neq \alpha$. We use generators of the ideal $I_{R(k)}$ corresponding to pairs of paths in $\TQ(k)$ with head at $k$. Consider paths of length two as in Figure~\ref{fig:relations1}, where for now we substitute $\rho(k)$ and $\rho(e_\delta)$ in place of $\rho_1$ and $\rho_2$.
\begin{figure}[!ht]
   \centering
		\begin{tikzpicture}[xscale=2,yscale=2]
		\tikzset{>=latex}
		
		\node (A) at (0,0) {$k-\du(\rho_1)-\du(\rho_2)$};
		\node (B) at (2,0) {$k-\du(\rho_2)$};
		\node (C) at (1.3,0.8) {$k-\du(\rho_1)$};
		\node (D) at (3.3,0.8) {$k$};
		
		\draw [->][red] (A) -- node [above] {\textcolor{black}{$\rho_1$}} (B);
		\draw [->][red] (C) -- node [above] {\textcolor{black}{$\rho_1$}} (D);
		\draw [->][blue] (B) -- node [below right] {\textcolor{black}{$\rho_2$}} (D);
		\draw [->][blue] (A) -- node [above left] {\textcolor{black}{$\rho_2$}} (C);
        \end{tikzpicture}
        \caption{}
        \label{fig:relations1}
\end{figure}
 In this case, we claim that each vertex in Figure~\ref{fig:relations1} lies in the quiver $\TQ(k)$. Indeed, $a_{\rho(k)}^k\in \TQ(k)_1$, so its head $k$ and tail $k-e_\beta+e_\alpha$ lie in $\TQ(k)_0$; this implies $k_\beta>0$ and either $\alpha=0$ or $k_\alpha< s_\alpha-1$. Also, $k_\delta>0$ and either $\gamma=0$ or $k_\gamma=0$, so $k-\du(\rho(e_\delta))$ is equal to $k-e_\delta+e_\gamma$, which lies in the quiver $\TQ(k)$. For the fourth vertex in Figure~\ref{fig:relations1}, either:
 \begin{enumerate}
 \item[\one] $\gamma=\alpha=0$, giving $e_\gamma=e_\alpha=0$, and the inequalities $k_\beta, k_\delta>0$ imply that the fourth vertex $k-e_\beta-e_\delta$ lies in $\TQ(k)_0$ as claimed; or  
 \item[\two] $\gamma\neq \alpha$, and since $\gamma<\delta<\beta$, the fourth vertex $k-e_\beta+e_\alpha-e_\delta+e_\gamma$ lies in $\TQ(k)_0$ because $k_\beta, k_\delta>0$, either $\alpha=0$ or $k_\alpha< s_\alpha-1$ and either $\gamma=0$ or $k_\gamma=0$.
 \end{enumerate} 
Figure~\ref{fig:relations1} therefore determines a binomial in $I_{R(k)}$ which implies that
\[
w_{\rho(e_\delta)}^{k-\du(\rho(k))}w_{\rho(k)}^k=w_{\rho(k)}^{k-\du(\rho(e_\delta))}w_{\rho(e_\delta)}^k.
\]
 Our induction assumption gives $w_\rho^m = w_\rho$ for all $m<k$, and since $w_{\rho(e_\delta)}=1=w_{\rho(k)}^k$,  we have $1 = w_{\rho(k)}w_{\rho(e_\delta)}^k$. In particular, $w_{\rho(k)}\neq 0$ and 
 \[
 w_{\rho(e_\delta)}^k=(w_{\rho(k)})^{-1}
 \]
 which establishes equation \eqref{eqn:wrhoN} for $\rho=\rho(e_\delta)$ when $\gamma=\alpha=0$ or $\gamma\neq \alpha$. 
 
 \medskip
\indent \textsc{Step 2:} Show that equation \eqref{eqn:wrhoN} holds for $\rho=\rho(e_\delta)$ when $\gamma=\alpha\neq 0$. Since $k_\alpha=k_\gamma=0$, the method from \textsc{Step 1} applies verbatim unless $s_\gamma=2$. In this case, define $0\leq \varepsilon < \gamma$ by 
\[
\varepsilon:=\tl(a_{\rho(e_\gamma)}),
\]
giving $\du(\rho(e_\gamma))=e_\gamma-e_\varepsilon$. Consider paths of length three as in Figure~\ref{fig:relations2}, where for now we substitute $\rho(k)$, $\rho(e_\delta)$ and $\rho(e_\gamma)$ in place of $\rho_1$, $\rho_2$ and $\rho_3$.  
 \begin{figure}[!ht]
   \centering
		\begin{tikzpicture}[xscale=2,yscale=2]
		\tikzset{>=latex}
		
		\node (A) at (0,0) {$k-\sum_{i=1}^3 \du(\rho_i)$};
		\node (B) at (2,0) {$k-\du(\rho_2)-\du(\rho_3)$};
		\node (C) at (4,0) {$k-\du(\rho_2)$};
		\node (D) at (1.3,0.8) {$k-\du(\rho_1)-\du(\rho_3)$};
		\node (E) at (3.3,0.8) {$k-\du(\rho_1)$};
		\node (F) at (5.3,0.8) {$k$};
		
		\draw [->][red] (A) -- node [above] {\textcolor{black}{$\rho_1$}} (B);
		\draw [->][black!35!green] (B) -- node [above] {\textcolor{black}{$\rho_3$}} (C);
		\draw [->][blue] (C) -- node [below right] {\textcolor{black}{$\rho_2$}} (F);
		\draw [->][blue] (A) -- node [above left, xshift=-2mm, yshift=-2mm] {\textcolor{black}{$\rho_2$}} (D);
		\draw [->][black!35!green] (D) -- node [above] {\textcolor{black}{$\rho_3$}} (E);
		\draw [->][red] (E) -- node [above] {\textcolor{black}{$\rho_1$}} (F);
		\end{tikzpicture}
        \caption{}
        \label{fig:relations2}
\end{figure}
 Again, we claim that each vertex in Figure~\ref{fig:relations2} lies in the quiver $\TQ(k)$; the proof is similar to that from \textsc{Step 1} (here, minimality of $\delta$ implies $\varepsilon=0$ or $k_\varepsilon=0$, and we use the inequalities $\varepsilon<\gamma<\delta<\beta$). Thus we obtain a binomial in $I_{R(k)}$ which, applying the inductive assumption $w_\rho^m=w_\rho$ for all $m<k$, gives 
 \[
 w_{\rho(e_\delta)}w_{\rho(e_\gamma)}w_{\rho(k)}^k=w_{\rho(k)}w_{\rho(e_\gamma)}w_{\rho(e_\delta)}^k.
 \]
 Since $w_{\rho(e_\delta)}=w_{\rho(e_\gamma)}=w^k_{\rho(k)} = 1$, we have $w_{\rho(k)}\neq 0$ and $w_{\rho(e_\delta)}^k=(w_{\rho(k)})^{-1}$ which implies that equation \eqref{eqn:wrhoN} holds for $\rho=\rho(e_\delta)$.
 
 \medskip

\indent \textsc{Step 3:} Show that equation \eqref{eqn:wrhoN} holds for all $\rho\in \Sigma(1)$. Consider any arrow $a_\rho^k$ in $\TQ$ with head at $k$. The vertices
\[
 \lambda:=\tl(a_\rho) \quad\text{and}\quad \mu:=\hd(a_\rho) 
\]
 satisfy $\du(\rho)=e_\mu-e_\lambda$ with $0\leq \lambda<\mu\leq \p$. We proceed using the approach from \textsc{Steps 1-2}:
\begin{enumerate}
\item[\one] If $\mu\neq \beta$, then we substitute $\rho$ and $\rho(k)$ in place of $\rho_1$ and $\rho_2$ in Figure~\ref{fig:relations1} as in \textsc{Step 1}, unless $\lambda=\alpha\neq 0$ and $s_\alpha=2$ in which case we substitute $\rho(e_\alpha)$ in place of $\rho_3$ in Figure~\ref{fig:relations2} as in \textsc{Step 2}. In either case, we obtain an equation relating components of $w_k$ which, after applying the inductive hypothesis if necessary, becomes
\[
w_{\rho(k)}w_\rho^k=w_\rho w_{\rho(k)}^k. 
\]
\textsc{Steps 1} and 2 established $w_{\rho(k)}\neq 0$, and $w_{\rho(k)}^k = 1$, so equation \eqref{eqn:wrhoN} holds.
 \item[\two] Otherwise, $\mu=\beta$. Substitute $\rho(e_\delta)$ and $\rho$ in place of  $\rho_1$ and $\rho_2$ in Figure~\ref{fig:relations1} as in \textsc{Step 1}, unless $\lambda=\gamma\neq0$ and $s_\gamma=2$ in which case we substitute $\rho(e_\gamma)$ in place of $\rho_3$ in Figure~\ref{fig:relations2} as in \textsc{Step 2}. As in part \one\ above, we obtain an equation which simplifies to 
\begin{equation}
\label{eqn:step3}
w_{\rho}^k=w_{\rho}w_{\rho(e_\delta)}^k.
\end{equation}
 \textsc{Steps 1} and \textsc{2} established $w_{\rho(e_\delta)}^k=(w_{\rho(k)})^{-1}$, so equation \eqref{eqn:wrhoN} follows.
\end{enumerate}
This completes the proof of equation \eqref{eqn:wrhoN} in \textsc{Case 1}.

 \medskip
\noindent \textsc{Case 2:} Suppose instead that $\delta=\beta$. If $k_\delta>1$ then the proof is identical to \textsc{Case 1}. If on the other hand $k_\delta=1$, then the vertex $k-\du(\rho(e_\delta))-\du(\rho(k)) = k-2e_\delta+e_\gamma+e_\alpha$ that plays a key role in \textsc{Case 1} does not lie in $\TQ(k)_0$. In the special case that $k=e_\delta$, making $k$ a vertex of the base quiver, then we have $w_\rho^k = w_\rho$ for all relevant $\rho\in \Sigma(1)$ and there is nothing to prove. If $k\neq e_\delta$, we introduce another useful vertex of the original quiver: let $\xi$ be minimal such that $\delta<\xi\leq \p$ and $k_\xi\neq 0$, and define $0\leq \eta <\xi$ by setting
 \[
\eta:= \tl(a_{\rho(e_\xi)})
 \]
 giving $\du(\rho(e_\xi))=e_\xi-e_\eta$. We treat the cases $\eta\neq \delta$ and $\eta=\delta$ separately.
 
 \medskip\indent \textsc{Subcase 2A}: If $\eta\neq \delta (=\beta)$, then either $\eta=0$ or $k_\eta=0$, so $k-\du(\rho(e_\xi)) = k-e_\xi+e_\eta$ is a vertex of $\TQ(k)_0$. We may now proceed just as in \textsc{Case 1} except that $\rho(e_\xi)$ replaces $\rho(e_\delta)$ throughout (so $\xi$ and $\eta$ replace $\delta$ and $\gamma$ respectively).
 
\medskip
\indent \textsc{Subcase 2B:} Suppose instead that $\eta=\delta (=\beta)$. We've already reduced to the case $k_\delta=1$. If $s_\delta>2$ then once again, $k-\du(\rho(e_\xi)) = k-e_\xi+e_\delta$ is a vertex of $\TQ(k)_0$ and we proceed as in \textsc{Case 1} with $\rho(e_\xi)$ replacing $\rho(e_\delta)$ throughout. If $s_\delta=2$, then we proceed as follows:

\smallskip
\indent \textsc{Step 1:} Show that $w_{\rho(k)}\neq 0$. If $\gamma\neq \alpha$ or $\gamma=\alpha=0$, then we use Figure~\ref{fig:relations2} with $\rho_1=\rho(k)$, $\rho_2=\rho(e_\delta)$ and $\rho_3=\rho(e_\xi)$ to obtain the equation
\[
1=w_{\rho(e_\delta)}w_{\rho(e_\xi)}w_{\rho(k)}^k=w_{\rho(k)}w_{\rho(e_\xi)}w_{\rho(e_\delta)}^k
\]
which gives $w_{\rho(k)}\neq 0$. Otherwise, $\gamma=\alpha\neq 0$, giving $\du(\rho(k))=e_\delta-e_\gamma = \du(\rho(e_\delta))$. It may be that 
$\rho(k)=\rho(e_\delta)$, in which case $w_{\rho(k)}=w_{\rho(e_\delta)}=1$ and hence $w_{\rho(k)}\neq 0$ as required. If $\rho(k)\neq\rho(e_\delta)$, then consider the pair of paths of length four as in Figure~\ref{fig:relations3}, where we substitute $\rho(k)$, $\rho(e_\delta)$, $\rho(e_\xi)$ and $\rho(e_\gamma)$ in place of $\rho_1,\dots,\rho_4$ (in fact, both paths pass through the same set of vertices in this case).
\begin{figure}[!ht]
   \centering
		\begin{tikzpicture}[xscale=2,yscale=2]
		\tikzset{>=latex}
		
		\node (A) at (0,0) {$k-\sum_{t=1}^4 \du(\rho_t)$};
		\node (B) at (2,0) {$k-\sum_{t=2}^4 \du(\rho_t)$};
		\node (C) at (4,0) {$k-\du(\rho_2)-\du(\rho_3)$};
		\node (D) at (1,0.8) {$k-\sum_{t\neq2} \du(\rho_t)$};
		\node (E) at (3,0.8) {$k-\du(\rho_1)-\du(\rho_3)$};
		\node (F) at (5,0.8) {$k-\du(\rho_1)$};
        \node (G) at (6,0) {$k-\du(\rho_2)$};
        \node (H) at (7,0.8) {$k$};
		
		\draw [->][red] (A) -- node [above] {\textcolor{black}{$\rho_1$}} (B);
		\draw [->][magenta] (B) -- node [above] {\textcolor{black}{$\rho_4$}} (C);
		\draw [->][blue] (A) -- node [above left, xshift=-2mm, yshift=-2mm] {\textcolor{black}{$\rho_2$}} (D);
		\draw [->][magenta] (D) -- node [above] {\textcolor{black}{$\rho_4$}} (E);
		\draw [->][black!35!green] (E) -- node [above] {\textcolor{black}{$\rho_3$}} (F);
        \draw [->][blue] (G) -- node [below right] {\textcolor{black}{$\rho_2$}} (H);
        \draw [->][black!35!green] (C) -- node [above] {\textcolor{black}{$\rho_3$}} (G);
        \draw [->][red] (F) -- node [above] {\textcolor{black}{$\rho_1$}} (H);
		\end{tikzpicture}
        \caption{}
        \label{fig:relations3}
\end{figure}

\noindent We obtain the equation
\[
1=w_{\rho(e_\delta)}w_{\rho(e_\gamma)}w_{\rho(e_\xi)}w_{\rho(k)}^k=w_{\rho(k)}w_{\rho(e_\gamma)}w_{\rho(e_\xi)}w_{\rho(e_\delta)}^k
\]
which gives $w_{\rho(k)}\neq 0$ and completes \textsc{Step 1}. 

\medskip
\indent \textsc{Step 2:} Show that equation \eqref{eqn:wrhoN} holds for all $\rho\in \Sigma(1)$. For any $a_\rho^k\in \TQ_1$, the vertices
\[
\lambda:=\tl(a_\rho) \quad\text{and}\quad \mu:=\hd(a_\rho) 
\]
 satisfy $\du(\rho)=e_\mu-e_\lambda$ with $0\leq \lambda<\mu\leq \p$.
\begin{enumerate}
\item[\one] If $\mu>\delta$, use Figure~\ref{fig:relations1} with $\rho_1=\rho(k)$ and $\rho_2=\rho$, unless $\lambda=\alpha\neq 0$ and $s_\alpha=2$ in which case use Figure~\ref{fig:relations2} with the addition of $\rho_3=\rho(e_\alpha)$. Either way, we obtain the equation $w_{\rho(k)}w_\rho^k=w_\rho w_{\rho(k)}^k$ which, since $w_{\rho(k)}\neq 0$ by \textsc{Step 1}, gives \eqref{eqn:wrhoN}.
\item[\two] If $\mu=\delta$, use Figure~\ref{fig:relations2} with $\rho_1=\rho(k)$, $\rho_2=\rho$ and $\rho_3=\rho(e_\xi)$ unless $\lambda=\alpha\neq 0$ and $s_\alpha=2$ in which case use Figure~\ref{fig:relations3} with the addition of $\rho_4=\rho(e_\alpha)$. Either way, we obtain $w_{\rho(k)}w_\rho^k=w_\rho w_{\rho(k)}^k$ which, since $w_{\rho(k)}\neq 0$ by \textsc{Step 1}, gives \eqref{eqn:wrhoN}.
\end{enumerate}
This concludes the proof in \textsc{Case 2}, and completes the proof of Theorem~\ref{thm:main}.
\end{proof}

\begin{remark}
 Our approach relies on the explicit description of the image of the morphism $f_E$ in Theorem~\ref{thm:main} as the GIT quotient $\VV(I_{\TQ})\git_\theta T$, see \cite[Theorem~1.1]{CS08}. We do not at present have a similar description in the non-toric setting.
\end{remark} 

\begin{example} We conclude with an example to illustrate the proof of Theorem~\ref{thm:main}. Let $Q$ and $\TQ$ be the quivers in Figure~\ref{fig:bothQuivers}, so $\p=3$. Suppose $k=(0,1,1)\in \TQ_0$, so $\delta=2$. The three arrows with head at $k$ have tails at $(1,1,0)$ (light blue), $(0,0,1)$ (red) and $(1,0,1)$ (blue),
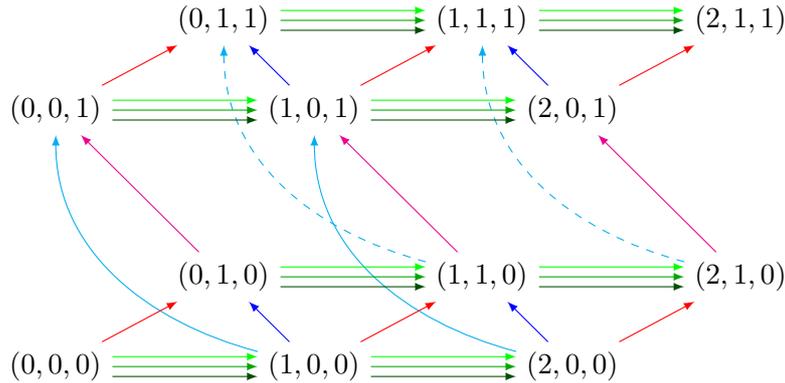
\begin{figure}[!ht]
   \centering
            \begin{tikzpicture}[xscale=1.7,yscale=1.7]
			\tikzset{>=latex}
			
			\node (A) at (0,0) {$(0,0,0)$};
			\node (B) at (2,0) {$(1,0,0)$};
			\node (C) at (1.3,0.7) {$(0,1,0)$};
			\node (D) at (0,2) {$(0,0,1)$};
			\node (E) at (4,0) {$(2,0,0)$};
			\node (F) at (3.3,0.7) {$(1,1,0)$};
			\node (G) at (2,2) {$(1,0,1)$};
			\node (H) at (1.3,2.7) {$(0,1,1)$};
			\node (I) at (5.3,0.7) {$(2,1,0)$};
			\node (J) at (4,2) {$(2,0,1)$};
			\node (K) at (3.3,2.7) {$(1,1,1)$};
			\node (L) at (5.3,2.7) {$(2,1,1)$};
			
			\draw [->][cyan] (B) to [out=165, in=270] (D);
			\draw [->][cyan] (E) to [out=165, in=270] (G);
			\draw [dashed, ->][cyan] (F) to [out=165, in=270] (H);      
         	\draw [dashed, ->][cyan] (I) to [out=165, in=270] (K);
			
			\draw [->][magenta] (C) -- (D);
			\draw [->][magenta] (F) -- (G);
			\draw [->][magenta] (I) -- (J);
			
			\draw [->][red] (A) -- (C);
			\draw [->][red] (B) -- (F);
			\draw [->][red] (E) -- (I);
			\draw [->][red] (D) -- (H);
			\draw [->][red] (G) -- (K);
			\draw [->][red] (J) -- (L);
			
			\draw [->][blue] (B) -- (C);
			\draw [->][blue] (E) -- (F);
			\draw [->][blue] (G) -- (H);
			\draw [->][blue] (J) -- (K);
			
			\draw [->][green] (A.10) -- (A.10-|B.west);
			\draw [->][black!35!green] (A) -- (B);
			\draw [->][black!70!green] (A.-10) -- (A.-10-|B.west);
			
			\draw [->][green] (B.10) -- (B.10-|E.west);
			\draw [->][black!35!green] (B) -- (E);
			\draw [->][black!70!green] (B.-10) -- (B.-10-|E.west);
			
			\draw [->][green] (C.10) -- (C.10-|F.west);
			\draw [->][black!35!green] (C) -- (F);
			\draw [->][black!70!green] (C.-10) -- (C.-10-|F.west);   
            
			\draw [->][green] (F.10) -- (F.10-|I.west);
			\draw [->][black!35!green] (F) -- (I);
			\draw [->][black!70!green] (F.-10) -- (F.-10-|I.west);
			
			\draw [->][green] (D.10) -- (D.10-|G.west);
			\draw [->][black!35!green] (D) -- (G);
			\draw [->][black!70!green] (D.-10) -- (D.-10-|G.west);
			
			\draw [->][green] (G.10) -- (G.10-|J.west);
			\draw [->][black!35!green] (G) -- (J);
			\draw [->][black!70!green] (G.-10) -- (G.-10-|J.west);
			
			\draw [->][green] (H.10) -- (H.10-|K.west);
			\draw [->][black!35!green] (H) -- (K);
			\draw [->][black!70!green] (H.-10) -- (H.-10-|K.west);
			
			\draw [->][green] (K.10) -- (K.10-|L.west);
			\draw [->][black!35!green] (K) -- (L);
			\draw [->][black!70!green] (K.-10) -- (K.-10-|L.west);
			
			\end{tikzpicture}
           \caption{The tilting quiver of $Q$ from Figure~\ref{fig:bothQuivers}}
           \label{fig:examplequiver}
\end{figure}
 and we label the corresponding rays $\rho_1, \rho_2$ and $\rho_3$ respectively. We now illustrate in two different situations why $w_{\rho(k)}\neq 0$ and why the equation $w_{\rho}^k=w_{\rho}(w_{\rho(k)})^{-1}$ from \eqref{eqn:wrhoN} holds for $\rho=\rho_1, \rho_2, \rho_3$.
\begin{enumerate}
\item  Suppose that $\rho(k)=\rho_1$. Then $\beta=3$ and $\alpha=1$ (see Figure~\ref{fig:basequiverQ}), and $w^k_{\rho_1}=1$. Suppose $\rho(e_\delta)=\rho(e_2)=\rho_2$ so that $\gamma=0$ and $w_{\rho_2}=1$. This is an example of \textsc{Case 1} as $\delta<\beta$, and since $\gamma=0$ we require only \textsc{Step 1}. In this case Figure~\ref{fig:relations1} becomes
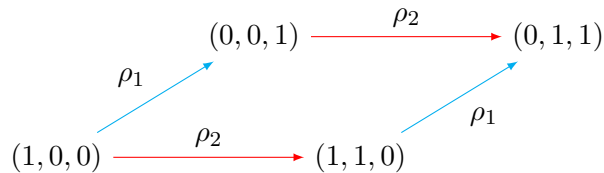
\begin{figure}[!ht]
   \centering
		\begin{tikzpicture}[xscale=2,yscale=2]
		\tikzset{>=latex}
		
		\node (A) at (0,0) {$(1,0,0)$};
		\node (B) at (2,0) {$(1,1,0)$};
		\node (C) at (1.3,0.8) {$(0,0,1)$};
		\node (D) at (3.3,0.8) {$(0,1,1)$};
		
		\draw [->][red] (A) -- node [above] {\textcolor{black}{$\rho_2$}} (B);
		\draw [->][red] (C) -- node [above] {\textcolor{black}{$\rho_2$}} (D);
		\draw [->][cyan] (B) -- node [below right] {\textcolor{black}{$\rho_1$}} (D);
		\draw [->][cyan] (A) -- node [above left] {\textcolor{black}{$\rho_1$}} (C);
        \end{tikzpicture}
        \label{fig:example1}
\end{figure}

\noindent and the relation gives the equation $w_{\rho_2} w^k_{\rho_1}=w_{\rho_1} w^k_{\rho_2}$. Moreover, $w_{\rho_2} =1=w^k_{\rho_1}$ implies $w_{\rho_1}\neq0$ and $w^k_{\rho_2}=(w_{\rho_1})^{-1}$ which establishes \eqref{eqn:wrhoN} for $\rho=\rho_1, \rho_2$. The remaining arrow $a^k_{\rho_3}$ with head at $k$ requires \textsc{Step 3}, and in this case for $\rho=\rho_3$ we have $\mu=2$ and $\lambda=1$. Since $\mu\neq \beta$ and $s_\alpha=s_1\neq2$, we require \textsc{Step 3}(i) to deduce $w_{\rho_3} w^k_{\rho_1}=w_{\rho_1} w^k_{\rho_3}$. This implies $w^k_{\rho_3}=w_{\rho_3} (w_{\rho_1})^{-1}$, establishing \eqref{eqn:wrhoN} for $\rho=\rho_3$.
\item Suppose $\rho(k)=\rho_3$, so $\beta=2$, $\alpha=1$ and $w^k_{\rho_3}=1$. Suppose that $\rho(e_2)=\rho_2$, so $\gamma=0$ and $w_{\rho_2}=1$. Since $\delta=\beta$ and $k_\delta=k_2=1$, this is an example of \textsc{Case 2}. Since $k\neq e_2$, we compute $\xi=3$. Write $\rho_4$ for the label of the pink arrow with head at $(0,0,1)$ and tail at $(0,1,0)$, and  suppose $\rho(e_3)=\rho_4$. Then $\eta=\tl(a_{\rho_4})=2$ and $w_{\rho_4}=1$. Since $\eta=\delta$ and $s_\delta=2$, we require \textsc{Subcase~2B}. Following \textsc{Step 1}, since $\gamma=0$ we use Figure~\ref{fig:relations2} as shown below.
\begin{figure}[!ht]
   \centering
		\begin{tikzpicture}[xscale=2,yscale=2]
		\tikzset{>=latex}
		
		\node (A) at (0,0) {$(1,0,0)$};
		\node (B) at (2,0) {$(0,1,0)$};
		\node (C) at (4,0) {$(0,0,1)$};
		\node (D) at (1.3,0.8) {$(1,1,0)$};
		\node (E) at (3.3,0.8) {$(1,0,1)$};
		\node (F) at (5.3,0.8) {$k$};
		
		\draw [->][blue] (A) -- node [above] {\textcolor{black}{$\rho_3$}} (B);
		\draw [->][magenta] (B) -- node [above] {\textcolor{black}{$\rho_4$}} (C);
		\draw [->][red] (C) -- node [below right] {\textcolor{black}{$\rho_2$}} (F);
		\draw [->][red] (A) -- node [above left, xshift=-2mm, yshift=-2mm] {\textcolor{black}{$\rho_2$}} (D);
		\draw [->][magenta] (D) -- node [above] {\textcolor{black}{$\rho_4$}} (E);
		\draw [->][blue] (E) -- node [above] {\textcolor{black}{$\rho_3$}} (F);
        \end{tikzpicture}
\end{figure}
 This yields the equation $w_{\rho_2} w_{\rho_4} w^k_{\rho_3}=w_{\rho_3} w_{\rho_4} w^k_{\rho_2}$ which simplifies to $1=w_{\rho_3} w^k_{\rho_2}$, giving $w_{\rho_3}\neq0$ as required. \textsc{Step~2} of \textsc{Subcase~2B} establishes \eqref{eqn:wrhoN} for $\rho=\rho_1,\rho_2,\rho_3$: we already know this for $\rho=\rho_3$ by assumption; the case $\rho=\rho_2$ is provided by \textsc{Step~1} since $w^k_{\rho_2}=(w_{\rho_3})^{-1}$; and the case $\rho=\rho_1$ is a simple application of \textsc{Step 2}$\one$, where we apply Figure~\ref{fig:relations1} to the rectangle with  vertices $(2,0,0),(1,1,0),(1,0,1),k$ and arrows labelled $\rho_1$ and $\rho_3$.
\end{enumerate}
\end{example}

\bibliographystyle{alpha}

\end{document}